\documentclass[11pt]{amsart}
\usepackage{amsmath,amssymb,amsfonts,amsthm,amscd,indentfirst}
\usepackage{amsmath,amssymb,amsfonts,amsthm,amscd}
\usepackage{indentfirst}
\usepackage{hyperref}

\allowdisplaybreaks[4]
\numberwithin{equation}{section}
\newtheorem{lemma}{Lemma}[section]

\newtheorem{proposition}[lemma]{Proposition}
\newtheorem{theorem}[lemma]{Theorem}

\newtheorem{remark}{Remark}

\numberwithin{equation}{section}
\usepackage{enumitem}%

\begin{document}
\title[Shrinking Gradient Ricci Solitons]
{A rigidity theorem for codimension one shrinking gradient Ricci solitons in $\mathbb R^{n+1}$}
\date{\today}
\author{Pengfei Guan, Peng Lu and Yiyan Xu}
\address{Department of Mathematics and Statistics,
McGill University,
Montreal, Quebec H3A 0B9,  Canada}
\email{guan@math.mcgill.ca}
\address{Department of Mathematics,
University of Oregon, Eugene, OR 97403,
USA}
\email{penglu@uoregon.edu}
\address{Department of Mathematics, Nanjing
University, Nanjing, China, 210093 and
Department of Mathematics and Statistics,
McGill University,
Montreal, Quebec H3A 0B9,  Canada}
	 \email{xuyiyan@math.pku.edu.cn}
\thanks{Research of P.G. is supported in part by NSERC Discovery Grant.}

\begin{abstract}
We prove a splitting theorem for complete gradient  Ricci soliton with nonnegative curvature and establish a rigidity theorem for codimension one complete shrinking gradient Ricci soliton in $\mathbb R^{n+1}$ with nonnegative Ricci curvature.
\end{abstract}
\maketitle

\section{Introduction}

A complete Riemannian metric $g$ on a smooth manifold $M^n$ is called a gradient Ricci soliton (GRS)
(Hamilton \cite{Hamilton93}, Perelman \cite{Perelman02}) if there exists a smooth function $f$ on $M$ such that
\begin{equation}\label{GraShrRicciSolitonEqu1}
\operatorname{Ric} + \operatorname{Hess}f=\frac{\lambda}{2}g,
\end{equation}
where $\lambda\in \mathbb{R}$. Below we assume that $\lambda=1,~0,~ \hbox{or}~ -1$; these
cases correspond to the GRS of shrinking, steady, or expanding type, respectively.

The classification of GRS,  especially the noncompact shrinking GRS, has been a subject of interest to
 many people. By the work of Perelman \cite{Perelman03}, Ni and Wallach \cite{NW08} and Cao, Chen, and Zhu \cite{CCZ08}, any 3-dimensional complete noncompact nonflat shrinking GRS must be the round cylinder
 $S^2\times \mathbb{R}$ or its $\mathbb{Z}_2$ quotient. Naber \cite{Naber10} proved that four dimensional
 complete noncompact shrinking GRS with bounded nonnegative curvature operator are finite quotients
  of generalized cylinders $S^2 \times \mathbb{R}^2$ or $S^3 \times \mathbb{R}$.
 Note that there are several rigidity results for higher dimensional complete noncompact
 shrinking GRS under various geometric assumptions \cite{NW08, Zhang09, PW10, CWZ11, MS13, Cai13}.

On the other hand,  Feldman, Ilmanen and Knopf \cite{FIK03} constructed $U(n)-$invariant shrinking
 K\"{a}hler GRS on the holomorphic line bundles $\mathcal{O}(-k)$ , $1\leq k\leq n$, over $P^{n-1}$,
 $n\geq 2$. Their examples are cone-like at infinity, and have Euclidean volume growth, positive scalar
 curvature and quadratic curvature decay.
However the Ricci curvature of these examples changes signs, more precisley the
 Ricci curvature is negative along the vertical (fiber)  direction  and  positive along horizontal direction.
  We do not know whether there is any nontrivial (the universal cover does not split) example of
  complete noncompact nonflat shrinking GRS with nonnegative Ricci curvature.

\medskip

The constant rank theorem is a powerful tool in the study of convex properties of
solutions of nonlinear differential equations \cite{CF85, SWYY, BG09}. In this paper we first establish
a constant rank theorem for Ricci tensor and for curvature operator of
 GRS (shrinking, steady, or expanding) and the corresponding splitting property of the GRS in section \ref{SplitThmPf}.

\begin{theorem}\label{SplGSRSThm1} Let $(M^n,g,f)$ be a GRS satisfying \eqref{GraShrRicciSolitonEqu1}.
 \begin{enumerate}[label=(\Roman*)]
   \item \label{SplStrRicThm}  If $g$ has nonnegative sectional curvature, then the rank of Ricci curvature is constant.
   Thus, either Ricci curvature is strictly positive  or the universal covering $(\widetilde{M},\tilde{g})=(N,h)\times \mathbb{R}^{n-k}$ splits isometrically and $(N,h)$ has
   strictly positive Ricci curvature;
   \item \label{SplStrCurOpeThm}  If $g$ has nonnegative curvature operator, then the rank of curvature operator is constant.  Thus, either the curvature operator is strictly positive or the universal covering $(\widetilde{M},\tilde{g})=(N,h)\times \mathbb{R}^{n-k}$ splits isometrically and $(N,h)$ has strictly positive curvature operator.
 \end{enumerate}
\end{theorem}

Note that since Ricci solitons satisfy Ricci flow, the splitting theorem can be obtained from the maximum principle for tensors in the parabolic setting. The maximum principle of this type was first proved by Hamilton for compact manifolds \cite{Hamilton86}, we also refer \cite{Ni04} for the corresponding result for complete noncompact manifolds under certain growth condition on tensors.

\vskip .1cm
With the help of Theorem \ref{SplGSRSThm1}, to classify shrinking GRS with nonnegative curvature operator, one
only needs to consider  GRS with positive curvature operator.
Since compact GRS with positive curvature operator must be of constant curvature, therefore, to prove the
rigidity of the complete shrinking GRS with nonnegative curvature operator one only needs to rule out
the noncompact shrinking GRS with positive curvature operator.
Note that in the K\"{a}hler case, Ni \cite{NL05} proved that a complete shrinking GRS with positive bisectional curvature must be compact;  therefore the GRS is isometric to
complex projective space $\mathbb{P}^m$ by the Mori-Siu-Yau theorem.

In the second part of this paper, we consider codimension one shrinking GRS $(M^n,g)$
isometrically embedded in $\mathbb R^{n+1}$.
If it has nonnegative Ricci curvature, then it has nonnegative curvature operator. By the classical
convexity theorem of Sacksteder-van Heijenoort, $M$ is a convex hypersurface.
The main result of this paper is the following.

\begin{theorem}\label{GSRSHypCyl} A complete codimension one shrinking GRS isometrically embedded
in $\mathbb R^{n+1}$ with positive Ricci curvature must be compact. As a consequence,
if $(M^n,g, f)$ is a complete shrinking GRS isometrically embedded in $R^{n+1}$ with nonnegative
Ricci curvature, then $(M,g)$ is a generalized cylinder $S^{k}\times \mathbb{R}^{n-k},
\, 2 \leq k\leq n$.
\end{theorem}

The main ingredients of our proof are the estimate of the mean curvature and the
eigenvalue estimate of a generalized Cheng-Yau operator associated to shrinking GRS. These results
will be proved in section \ref{MeaCurGroEstSec} and section \ref{StrGenCyl}, respectively.
\\

\noindent\textbf{Acknowledgement}:  P.L. is partially supported by a Simons grant.
The work was done when Y.X. was supported by CRC Postdoctoral Fellowship at McGill University.

\section{Preliminaries of Gradient Ricci Soliton}

We collect some well known identities for gradient GRS below, they  can be found in
\cite{Hamilton93}. For a GRS satisfying \eqref{GraShrRicciSolitonEqu1}, let $\{e_i\}$ is an orthonormal basis of $TM$, the Ricci curvature
satisfies the following formula:
\begin{gather}
\nabla R= 2 \operatorname{Ric}(\nabla f), \label{RicSolIde1}\\
\Delta_f R_{ij}= 2\lambda R_{ij}-2\overset{\circ}{R}(\operatorname{Ric})_{ij}, \label{LapRicSolEqu}
\end{gather}
where
\begin{align*}
\Delta_f =\nabla_{ii}^2  - \nabla_i f \nabla_i \quad\text{ and }\quad
 \overset{\circ}{R}(\operatorname{Ric})_{ij}=\sum_kR(e_i,\operatorname{Ric}(e_k),e_j,e_k).
\end{align*}

Moreover,  after identifying $\Lambda^2T_xM$ with $\mathfrak{so}(T_xM)$,
the curvature  operator is symmetric endomorphism $\mathfrak{R}\in S^2(\mathfrak{so}(T_xM) )$,
\[\mathfrak{R}_{\alpha\beta}=R_{ijkl}\phi_\alpha^{ij}\phi_\beta^{kl},\quad\phi_\alpha=\phi_\alpha^{ij}e_i\wedge e_j\in\Lambda^2T_xM .\]
For a GRS satisfying \eqref{GraShrRicciSolitonEqu1}, the curvature operator satisfies the following formula:
\begin{equation}\label{LapCurOpeSolEqu}
 \Delta_f\mathfrak{R}_{\alpha\beta}=2\lambda\mathfrak{R}_{\alpha\beta}-\mathfrak{R}^2_{\alpha\beta}-
 \mathfrak{R}^\sharp_{\alpha\beta},
\end{equation}
where
\begin{equation}\label{PosQuaCurOpe1}
\begin{split}
  \langle\mathfrak{R}^\sharp(\phi_\alpha),\phi_\alpha\rangle&=\langle{\rm ad}\circ(\mathfrak{R}\wedge \mathfrak{R})\circ {\rm ad}^*(\phi_\alpha),\phi_\alpha\rangle \\
  &=\sum_{\beta,\gamma}\langle[\mathfrak{R}(\phi_\beta),\mathfrak{R}(\phi_\gamma)],\phi_\alpha
  \rangle\langle[\phi_\beta,\phi_\gamma],\phi_\alpha\rangle,
\end{split}
\end{equation}
here ${\rm ad}: \Lambda^2(\mathfrak{so}(T_xM))\rightarrow \mathfrak{so}(T_xM), \phi\wedge\varphi\mapsto{\rm ad}(\phi\wedge\varphi)=[\phi,\varphi]$ is the adjoint representation.

The following identities involving the curvature and potential function are satisfied for shrinking GRS
\cite{Hamilton93}\label{IdeCGSRS1},
\begin{gather}
  R+\Delta f= \frac{n}{2}, \label{RicSolIde0}   \\
 R+|\nabla f|^2-f=C_0(=0)\label{RicSolIde2}
\end{gather}
for some constant $C_0$ (By adding a constant to $f$, we assume $C_0=0$ below).

The behavior of the potential function plays an important role in understanding the structure of
shrinking GRS \cite{Perelman03, NW08, CCZ08}. The following estimates are due to \cite{Perelman02, CZ10, HM11}.

\begin{proposition}\label{ProPotGSRS} Let $(M^n,g,f)$ be complete non-compact
shrinking GRS satisfying \eqref{GraShrRicciSolitonEqu1}. Let $x_0\in M$ be the point
such that $f(x_0) =\min_{x \in M} f(x)$. Then the potential
function $f$ satisfies the estimates
\begin{gather}
 \frac{1}{4}(r(x)-5n)_+^2\leq f(x)\leq\frac{1}{4}(r(x)+\sqrt{2n})^2,\label{PotQuaGroPotEst1}  \\
|\nabla f|\leq \frac{1}{2} (r(x)+2\sqrt{f(x_0)}),\label{PotGraGroPotEst1}
\end{gather}
where $r(x)=d(x,x_0)$ is the distance function and $a_+:= \max\{a, 0\}$.
Consequently,
\begin{equation}\label{FunExpGrwFin1}
\int_M|u|e^{-f}d\mu <+\infty
\end{equation}
for any continuous function $u$ on $M$ satisfying $|u(x)|\leq Ae^{\alpha r^{2}(x)}$ where $0\leq \alpha<\frac{1}{4}$
and $A>0$.
In particular,  the weighted volume of $M$ $\int_Me^{-f}d\mu$ is finite.

\end{proposition}

Chen \cite{Chen09} proved that the scalar curvature of complete ancient solution of Ricci flow is always nonnegative.  For complete non-flat shrinking GRS $(M^n,g,f)$ the asymptotic estimates
for the potential function $f$ also controls the curvature growth
rates. In particular, \eqref{RicSolIde2} and \eqref{PotQuaGroPotEst1} imply that the scalar curvature grows
at most quadratically,
 \begin{equation}\label{PloGroScaCur1}
 0\leq R(x)\leq \frac{1}{4}(r(x)+\sqrt{2n})^2.
  \end{equation}
When $(M,g,f)$ is assumed to have nonnegative Ricci survature,
by \eqref{RicSolIde1} we have
\[
\langle \nabla R, \nabla f\rangle= 2 \operatorname{Ric}(\nabla f,\nabla f )\geq0,
\]
thus scalar curvature $R$ is increasing along the gradient flow of potential $f$. It is established by Ni \cite[Proposition 1.1]{NL05} that there exists a $\delta_0 =\delta(M) \in (0, 1)$ such that
\begin{equation}\label{ScaStrPosBou1}
R\geq \delta_0>0.
\end{equation}
Combining $|\operatorname{Ric}|^2\leq R^2$ with \eqref{RicSolIde1}, we conclude that the gradient of scalar curvature
 grows at most polynomial fast,
\begin{equation}\label{PloGroGraScaCur1}
|\nabla R|^2=|2 \operatorname{Ric}(\nabla f)|^2\leq 4R^2|\nabla f|^2.
\end{equation}
By \eqref{RicSolIde0}, \eqref{RicSolIde1}, \eqref{RicSolIde2}, together with the Bochner identity,
\begin{equation}
\begin{split}
\Delta R&=\Delta(f-|\nabla f|^2)\\
&=\Delta f-2\big(|\nabla^2f|^2+\langle\nabla\Delta f, \nabla f\rangle+ \operatorname{Ric}(\nabla f, \nabla f)\big)\\
&= \frac{n}{2}-R -2|\operatorname{Ric}-\frac{1}{2}g|^2-2\langle\nabla (\frac{n}{2}-R), \nabla f\rangle-2
\operatorname{Ric}(\nabla f, \nabla f)\\
&=\frac{n}{2}-R-2|\operatorname{Ric}-\frac{1}{2}g|^2+2 \operatorname{Ric}(\nabla f, \nabla f).
\end{split}
\end{equation}
Hence,
\begin{equation}\label{PolyGrowLapSca}
\begin{split}
 |\Delta R| &\leq \frac{n}{2}+|R|+2|\operatorname{Ric}-\frac{1}{2}g|^2+2|\operatorname{Ric}||\nabla f|^2\\
 &\leq n+2R^2+2R|\nabla f|^2.
\end{split}
\end{equation}

We will need the following classification result for compact GRS which follows from
 the works of Hamilton \cite{Hamilton82, Hamilton86} (dimensions three
and four) and B\"{o}hm and Wilking \cite{BW08} (dimensions $\geq 5$).
\begin{theorem}\label{GSRSPosCOCla}
A compact GRS with positive curvature operator must be a space form.
\end{theorem}

\section{Splitting Theorem of Gradient Ricci Soliton}\label{SplitThmPf}
In this section, we establish the constant rank Theorem \ref{SplGSRSThm1} for GRS with
nonnegative curvature via strong maximum principle.
We will show  that the distribution of the null space of the Ricci tensor is of
 constant dimension and is invariant under parallel translation.  That would yield a splitting
 theorem for GRS. Similar conclusion also holds for the curvature operator.

\subsection{Ricci curvature and constant ranking Theorem \ref{SplGSRSThm1}~\ref{SplStrRicThm}}\label{ProSplThmRic1}

Let $A=(a_{ij})_{n \times n}$ be a symmetric matrix. Define
\begin{equation}\label{DefSigk}
\det(I+tA)=\sum_{l=0}^n \sigma_l (A)t^l.
\end{equation}
Note that $\sigma_l (A)$ is a smooth function of variables $a_{ij}$.
When $A= \operatorname{diag}[\lambda_1,\cdots,\lambda_n]$ is a diagonal matrix, then  $\sigma_l (A)$
is the $l$-th elementary polynomial of $\lambda_1,\cdots,\lambda_n$.

If $A$ is any $n\times n$ symmetric matrix, we denote
\begin{equation}\label{DerSigEqu1}
\sigma_l^{ij}(A) := \frac{\partial\sigma_l(A)}{\partial a_{ij}},  \quad
\sigma_l^{ij,kl}(A) := \frac{\partial^2\sigma_l(A)}{\partial a_{ij}\partial a_{kl}}.
\end{equation}
In particular, we have
\[\sigma_1^{ij} (A)=\delta_{ij},\quad \sigma_2^{ij} (A)=(\sum_{k=1}^na_{kk})\delta_{ij}-a_{ij}.\]

We also denote by $(A|i)$ the $(n-1) \times(n-1)$ matrix obtained from $A$ by
deleting the $i$-th row and $i$-th column, and by $(A|ij)$ the $(n-2) \times(n-2)$ matrix obtained from
$A$ by deleting the $i$, $j$-th  row and $i$, $j$-th column.

The following two propositions (See Proposition \ref{DerSigRan1} and Proposition \ref{DerSigRan2} below) are well known (e.g., see \cite{BG09}).
\begin{proposition}\label{DerSigRan1}If $A$ is a
diagonal matrix. For any $l, i, j$ we have
\[\sigma_l^{ij} (A) =\left\{
                \begin{array}{ll}
                  \sigma_{l-1}(A|i), & \hbox{if $i=j$;} \\
                          0, & \hbox{otherwise.}
                \end{array}
              \right.
\]
and
 \[
\sigma_l^{ij,kl} (A)=\left\{
                \begin{array}{ll}
                  \sigma_{l-2}(A|ik), & \hbox{if $i=j$, $k=l$, $i\neq k$;} \\
                  -\sigma_{l-2}(A|ij), & \hbox{if $i=l$, $j=k$, $i\neq j$;} \\
                          0, & \hbox{otherwise.}
                \end{array}
              \right.
\]
\end{proposition}

Let $(M^n,g,f)$ be a GRS with nonnegative Ricci curvature as in Theorem \ref{SplGSRSThm1}.
In this case, we take $A=Ric$. We may assume that $ r := \min_{x\in M} \operatorname{rank} \operatorname{Ric}(x) <n$;
otherwise we have $\operatorname{Ric}>0$.
Let $x_0 \in M$ be a point such that $\operatorname{rank} \operatorname{Ric}(x_0)=r$.
Pick a small open neighborhood $\mathcal{O}$ of $x_0$. We define function $\phi$ on $\mathcal{O}$ by
\[
\phi (x) =\sigma_{r+1}(\operatorname{Ric} (x)).
\]
To prove Theorem \ref{SplGSRSThm1}~\ref{SplStrRicThm}, we first show that
there is a positive constant $C$ independent of $\phi$ such that on $\mathcal{O}$
\[
\Delta\phi(x) \leq C(\phi(x)+ |\nabla \phi (x) |).
\]

In the following we shall use the notations used in \cite{BG09}.
For any $x\in\mathcal{O}$, let $\lambda_1(x)\leq
\lambda_2(x)\leq \cdots \leq \lambda_n(x)$ be the eigenvalues of $\operatorname{Ric}(x)$. There is a positive constant $C_0>0$ depending $\mathcal{O}$, such that $\lambda_1(x)\leq \lambda_2(x)\leq \cdots \lambda_{n-r}(x) \leq \frac{C_0}{10^{100n}}$ and $C_0\leq \lambda_{n-r+1}(x)\leq \lambda_{n-r+2}(x)\leq \cdots \leq \lambda_n(x)$
for all $x\in \mathcal{O}$.
Let $G=\{n-r+1,n-r+2,\cdots,n\}$ and $B=\{1,\cdots,n-r\}$ be the ``good" and ``bad"
sets of indices for eigenvalues of $\operatorname{Ric}$, respectively.
Define diagonal matrix $\Lambda_G=\operatorname{diag}[0, \cdots,0, \lambda_{n-r+1},\lambda_{n-r+2},
\cdots,\lambda_{n}]$ and
$\Lambda_B=\operatorname{diag}[\lambda_{1},\cdots,\lambda_{n-r},0, \cdots, 0]$.  Use notation $h=O(k)$ if $|h(x)|\leq Ck(x)$ for $x\in \mathcal{O}$ with some positive constant $C$ under control. In particular,
$\lambda_i=O(\phi)$ for all $i\in B$, and
\begin{equation}
(\sum_{i\in B}\lambda_i)\sigma_{r}(\Lambda_G)=O(\phi). \label{ZeroOrdEqu1}
\end{equation}

Based on Proposition \ref{DerSigRan1}, with the notation as above, we have

\begin{proposition} \label{DerSigRan2} Let $A=Ric$ as above. Then we have that on $\mathcal{O}$
\[ \frac{\partial \sigma_{r+1}(A)}{\partial a_{ij}}=\left\{
                \begin{array}{ll}
                  \sigma_{r}(\Lambda_G)+O(\phi), & \hbox{if $i=j\in B$;} \\
                          O(\phi), & \hbox{otherwise.}
                \end{array}
              \right.
\]
and
\[
\frac{\partial^2\sigma_{r+1}(A) }{\partial a_{ij}\partial a_{kl}}=\left\{
                \begin{array}{ll}
                  \sigma_{r-1}(\Lambda_G|i)+O(\phi)=\frac{1}{\lambda_i}\sigma_{r}(\Lambda_G)+O(\phi), &
                  \hbox{if $i=j\in G,k=l\in B$;} \\
                   \sigma_{r-1}(\Lambda_G|k)+O(\phi)=\frac{1}{\lambda_k}\sigma_{r}(\Lambda_G)+O(\phi), &
                  \hbox{if $i=j\in B,k=l\in G$;} \\
                   \sigma_{r-1}(\Lambda_G)+O(\phi), & \hbox{if $i=j\in B,k=l\in B,i\neq k$;} \\
                   -\sigma_{r-1}(\Lambda_G|i)+O(\phi)= -\frac{1}{\lambda_i}\sigma_{r}(\Lambda_G)+O(\phi),
                   & \hbox{if $i=l\in G,j=k\in B$;} \\
                    -\sigma_{r-1}(\Lambda_G|j)+O(\phi)= -\frac{1}{\lambda_j}\sigma_{r}(\Lambda_G)+O(\phi),
                   & \hbox{if $i=l\in B,j=k\in G$;} \\
                   -\sigma_{r-1}(\Lambda_G)+O(\phi), & \hbox{if $i=l\in B,j=k\in B,i\neq j$;} \\
                          0, & \hbox{otherwise.}
                \end{array}
              \right.
\]
\end{proposition}

From Proposition \ref{DerSigRan2},  we compute the first derivative
\begin{gather}
  \phi_\alpha=\sigma_{r+1}^{ij}R_{ij,\alpha}= \sigma_{r}(\Lambda_G)\sum_{i\in B}R_{ii,\alpha}+O(\phi),
\label{FirOrdEqu1}
\end{gather}
and the second derivative
\begin{eqnarray}\label{SecOrdEqu1}
  \phi_{\alpha\beta}&=&\sigma_{r+1}^{ij}R_{ij,\alpha\beta}+\sigma_{r+1}^{ij,kl}R_{ij,\alpha}R_{kl,\beta}\nonumber\\
   &= &\sigma_{r}(\Lambda_G)\sum_{i\in B}R_{ii,\alpha\beta} + \sum_{i\in G}\sum_{k\in B}\frac{1}{\lambda_i}\sigma_{r}(\Lambda_G)R_{ii,\alpha}R_{kk,\beta} +  \sum_{i\in B}\sum_{k\in G }\frac{1}{\lambda_k}\sigma_{r}(\Lambda_G)R_{ii,\alpha}R_{kk,\beta}\nonumber\\
   &&+ \sum_{i,k\in B}\sigma_{r-1}(\Lambda_G)R_{ii,\alpha}R_{kk,\beta}- 2 \sum_{i\in G}\sum_{j\in B}\frac{1}{\lambda_i}\sigma_{r}(\Lambda_G)R_{ij,\alpha}R_{ji,\beta}\nonumber\\
   &&-\sum_{i,j\in B}\sigma_{r-1}(\Lambda_G)R_{ij,\alpha}R_{ji,\beta}+O(\phi).
   \end{eqnarray}
Take trace of \eqref{SecOrdEqu1} and using \eqref{FirOrdEqu1}, we get
\begin{equation}\label{SigRanDifIne1}
\begin{split}
\Delta\phi & =\sigma_{r}(\Lambda_G)\sum_{i\in B}\Delta R_{ii} +2\left ( \sum_{i \in G}
 \frac{1}{\lambda_i}R_{ii,\alpha} \right ) \phi_\alpha - 2 \sigma_{r}(\Lambda_G)\sum_{i\in
G}\sum_{j\in B}\frac{1}{\lambda_i}|\nabla R_{ij}|^2\\
&\quad +  \frac{\sigma_{r-1}(\Lambda_G)}{ \sigma_{r}^2(\Lambda_G)}
\phi_\alpha^2 -\sigma_{r-1}
(\Lambda_G) \sum_{i, j\in B}|\nabla R_{ij}|^2+O(\phi)\\
& =\sigma_{r}(\Lambda_G)\sum_{i\in B}\Delta R_{ii}  - 2 \sigma_{r}(\Lambda_G)\sum_{i\in
G}\sum_{j\in B}\frac{1}{\lambda_i}|\nabla R_{ij}|^2-\sigma_{r-1}
(\Lambda_G) \sum_{i, j\in B}|\nabla R_{ij}|^2\\
&\quad+O(\phi)+O(|\nabla\phi|).
\end{split}
\end{equation}
By identity \eqref{LapRicSolEqu},
\begin{equation}\label{SigRanDifIne2}
\begin{split}
\Delta\phi
=& \sigma_{r}(\Lambda_G)\sum_{i\in B}(\nabla_{\nabla f}R_{ii}+2\lambda R_{ii}
-2\overset{\circ}{R}(\operatorname{Ric})_{ii})\\
&- 2 \sigma_{r}(\Lambda_G)\sum_{i\in G}\sum_{j\in B}\frac{1}{\lambda_i}|\nabla R_{ij}|^2- \sigma_{r-1}(
\Lambda_G)\sum_{i,j\in B}|\nabla R_{ij}|^2+O(\phi)+O(|\nabla\phi|).
\end{split}
\end{equation}

To deal with the first term in the righthand side of \eqref{SigRanDifIne2}, by \eqref{ZeroOrdEqu1}
and  \eqref{FirOrdEqu1} we have
\begin{equation}\label{LapRic1Ord1}
\sigma_{r}(\Lambda_G)\sum_{i\in B}(\nabla_{\nabla f}R_{ii}+2\lambda R_{ii})=
O(\phi)+O(|\nabla\phi|).
\end{equation}
By the assumption of nonnegative sectional curvature \footnote{This is the only place that we need
nonnegative sectional curvature.},
\begin{equation}\label{LapRicQuaCur1}
\overset{\circ}{R}(\operatorname{Ric})_{ii}=\sum_kR(e_i, \operatorname{Ric}(e_k),e_i,e_k)
\geq \lambda_1\sum_kR(e_i,e_k,e_i,e_k)\geq 0.
\end{equation}
Combine \eqref{SigRanDifIne2}, \eqref{LapRic1Ord1}, and \eqref{LapRicQuaCur1},
\begin{equation}\label{SigRanDifIne3}
\begin{split}
\Delta\phi
&\leq C(\phi+ |\nabla\phi|) - 2 \sigma_{r}(\Lambda_G)\sum_{i\in G}\sum_{j\in B}\frac{1}{
\lambda_i}|\nabla R_{ij}|^2- \sigma_{r-1}( \Lambda_G)\sum_{i,j\in B}|\nabla R_{ij}|^2.
\end{split}
\end{equation}

Hence we have proved
\[\Delta\phi\leq C(\phi+|\nabla\phi|).\]
Since $\phi \geq 0$ on $\mathcal{O}$ and $\phi(x_0)=0$, it follows from the strong maximum principle that
$ \phi\equiv 0$ on $\mathcal{O}$. We conclude that  $\phi\equiv 0$ in $M$, i.e. $\operatorname{rank}
 \operatorname{Ric} \equiv r$.

Next we consider the null space of Ricci curvature $\operatorname{null} \operatorname{Ric}$.
It follows from \eqref{SigRanDifIne3} that
\begin{equation}
2 \sigma_{r}(\Lambda_G)\sum_{i\in G}\sum_{j\in B}\frac{1}{\lambda_i}|\nabla R_{ij}|^2+  \sigma_{r-1}(
\Lambda_G)\sum_{i,j\in B}|\nabla R_{ij}|^2\equiv 0.
\end{equation}
Hence for any $v\in \operatorname{null} \operatorname{Ric}$, $\nabla \operatorname{Ric}(v)=0$.
On the other hand, for any section $v \in \operatorname{null} \operatorname{Ric}$ and for any index $k$
we have
\[0=\nabla_{k}(R_{ij}v^i)=(\nabla_kR_{ij})v^j +R_{ij}\nabla_kv^i, \]
thus $R_{ij}\nabla_kv^i = -(\nabla_k R_{ij})v^j=0$.
This shows that $\nabla_k v \in \operatorname{null} \operatorname{Ric}$ and that $\operatorname{null}
\operatorname{Ric}$ is invariant under parallel translation.

Finally we show that the universal covering space of the GRS $(M,g)$ splits.
Since the distribution $\operatorname{null} \operatorname{Ric}$ is invariant under parallel translation,
 $\operatorname{null} \operatorname{Ric}$
is involutive. Let $(\operatorname{null} \operatorname{Ric})^\perp$ be the distribution that generated by orthogonal complements of $\operatorname{null} \operatorname{Ric}$. For any sections $X, Y\in(\operatorname{null}
\operatorname{Ric})^\perp$, $V\in \operatorname{null}\operatorname{Ric}$, then
\[g([X,Y],V)=g(\nabla_XY-\nabla_YX,V)=-g(Y,\nabla_XV)+g(X,\nabla_YV)=0.\]
Thus  the distribution $(\operatorname{null} \operatorname{Ric})^\perp$ is also involutive.
The classical deRham splitting theorem (see \cite[Theorem 10.43]{Besse87}) implies that $(M,g)$ locally splits.

Now consider the the universal covering space $(\widetilde{M},\tilde{g})$.
We denote by $L$ the leaf of the integral manifold of $\operatorname{null} \operatorname{Ric}$,
then $L$ is Ricci flat. By equation \eqref{GraShrRicciSolitonEqu1}, on every leaf,
${\operatorname{Hess}}f=\frac{1}{2} g $.
Consequently,  $L$ is isometric to $\mathbb{R}^{n-r}$. Hence $(\widetilde{M},\tilde{g})=(N,h)\times
\mathbb{R}^{n-r}$ split isometrically  along the null space of Ricci curvature, where $(N,h)$ has
 strictly positive Ricci curvature. We have finished the proof of Theorem \ref{SplGSRSThm1}~\ref{SplStrRicThm}.

\subsection{Curvature operator and constant rank Theorem \ref{SplGSRSThm1}(II)}
Similarly,  we can establish the constant rank theorem Theorem \ref{SplGSRSThm1}~\ref{SplStrCurOpeThm} for curvature operators.
We may assume that $r:= \min_{x\in M}{\rm rank\,}\mathfrak{R}(x)< \frac{n(n-1)}
{2}$. There is a point $x_0\in M$ such that ${\rm rank\,}\mathfrak{R}(x_0) =r$.
Pick an orthonormal basis $\{ e_i \}$ around $x_0$.  Let $\{\varphi_\alpha=\varphi_\alpha^{ij}e_i\wedge e_j\}$
be the eigenvectors of curvature operator $\mathfrak{R}$, i.e.
$\mathfrak{R}(\varphi_\alpha)=\lambda_\alpha\varphi_\alpha$.
Define $\phi=\sigma_{r+1}(\mathfrak{R})$.

Below we adopt notations similar to the ones used in section \ref{ProSplThmRic1}.
Using Proposition \ref{DerSigRan2} and equation \eqref{LapCurOpeSolEqu} and by a computation similar to
the derivation of (\ref{SigRanDifIne2}) we get
\begin{equation}\label{SigCRRanDifIne1}
\begin{split}
\Delta\phi& \leq C(\phi+ |\nabla \phi|)+ \sigma_{r}(\Lambda_G)\sum_{\alpha\in B} \left (\nabla_{\nabla f}\mathfrak{R}_{
\alpha\alpha}+  2\lambda \mathfrak{R}_{\alpha\alpha}-\mathfrak{R}^2_{\alpha\alpha} - \mathfrak{R}^\sharp_{
\alpha\alpha} \right )\\
&\quad -2 \sigma_{r}(\Lambda_G)\sum_{\alpha\in G}\sum_{\beta\in B}\frac{1}{\lambda_\alpha}|\nabla \mathfrak{R}_{\alpha\beta}|^2- \sigma_{r-1}(\Lambda_G)\sum_{\alpha, \beta\in B}|\nabla \mathfrak{R}_{
\alpha\beta}|^2 \\
&\leq C(\phi+ |\nabla \phi|) - 2 \sigma_{r}(\Lambda_G)\sum_{\alpha\in G}\sum_{\beta\in B}\frac{1}{\lambda_\alpha}
|\nabla \mathfrak{R}_{\alpha\beta}|^2- \sigma_{r-1}(\Lambda_G)\sum_{\alpha, \beta\in B}
|\nabla \mathfrak{R}_{\alpha\beta}|^2\\
&  \leq C(\phi+ |\nabla \phi|),
\end{split}
\end{equation}
in some small neighborhood $\mathcal{O}$ of $x_0$.
To get the second inequality above we have used the following
\begin{eqnarray*}
  \langle\mathfrak{R}^\sharp(\varphi_\alpha), \varphi_\alpha \rangle = \sum_{\alpha, \beta}\lambda_\beta\lambda_\gamma\langle[\varphi_\beta,\varphi_\gamma],\varphi_\alpha\rangle^2
  \geq 0,
\end{eqnarray*}
which follows from \eqref{PosQuaCurOpe1} and the
assumption of the nonnegative curvature operator.

Since $\phi\geq 0$, and $\phi(x_0)=0$, by applying the strong maximum principle to \eqref{SigCRRanDifIne1} we get
$\phi\equiv 0$.
We conclude that curvature operator $\mathfrak{R}$ has constant rank.

By a similar proof as for the Ricci curvature case the null space of $\mathfrak{R}$ is invariant under
parallel transilation.
Moreover, it follows from (\ref{LapCurOpeSolEqu}) that $\operatorname{null}\mathfrak{R} \subset
\operatorname{null}\mathfrak{R}^\sharp$. By \eqref{PosQuaCurOpe1} we have $ \langle[\varphi_\beta,
\varphi_\gamma], \phi\rangle=0$ for $\phi\in \operatorname{null} \mathfrak{R}$ and
for any $\beta\neq\gamma$ with $\lambda_\beta>0$ and $\lambda_\gamma>0$.
Since $\mathfrak{R}$ is a self-adjoint operator,
\[
\phi\in \operatorname{image} \mathfrak{R}\Leftrightarrow \langle \phi, \varphi_\alpha\rangle=0,
\quad \forall~ \alpha {~\rm with~}\lambda_\alpha=0.
\]
For any section $\phi, \omega \in \operatorname{image} \mathfrak{R}$, we have
\[
\langle[\phi,\omega],\psi\rangle=\sum_{\beta,\gamma} \langle\phi,\varphi_\beta\rangle \langle \omega,\varphi_\gamma\rangle\langle[\varphi_\beta,\varphi_\gamma],\psi \rangle
 =0, \quad \forall \psi\in \operatorname{null} \mathfrak{R},
 \]
hence $[\phi,\omega]\in \operatorname{image} \mathfrak{R}$.
This implies that the image of $\mathfrak{R}$ is a Lie subalgebra.
Ambrose-Singer theorem ensures that the Lie algebra  $\mathfrak{hol}(M,g)$ of Holonomy group
is reduced to a lower dimension, so by deRham splitting theorem (see \cite[Theorem 10.43]{Besse87}) the universal covering space
$(\widetilde{M},\tilde{g})$ is a Riemannian product. Since one of the product factor is flat from our
  construction, $(\widetilde{M},\tilde{g})=(N,h)\times \mathbb{R}^{n-m}$, $\frac{m(m-1)}{2}=r$, splits isometrically,
where  $(N,h)$ has strictly positive curvature operator.
The proof of Theorem \ref{SplGSRSThm1} is completed.

\section{Mean Curvature Growth Estimate}\label{MeaCurGroEstSec}

In this section, we establish the following  a priori interior estimate of the mean curvature for a
convex hypersurface in $\mathbb R^{n+1}$. As a consequence, we can control the mean curvature growth for embedded codimension one GRS in $\mathbb R^{n+1}$, see \eqref{meaIntWeiMea1} and \eqref{MeaIntWeiMea1}.

Let $X: M^n \rightarrow \mathbb{R}^{n+1}$ be a hypersurface with induced metric $g$ and (outer) unit normal $\nu$. Let $\{e_1,\cdots,e_n\}$ be a local orthonormal frame filed on $M$, then
\begin{equation}\label{FirDerPosNorHyp}
X_{ij} = -h_{ij}\nu, \quad  1\leq i, j\leq n,
\end{equation}
where  $h=(h_{ij})$ is the second fundamental form. Let $\sigma_k=\sigma_k(h)$ be the $k$-th elementary symmetric function of the eigenvalues of $h$. In particular, $H=\sigma_1(h)$ and $R=2\sigma_2(h)$ are the mean curvature and the scalar curvature respectively. If the scalar curvature of $M$ is positive, we take the unit normal $\nu$ such that $h$ lies in Garding's  $\Gamma_2$-cone. In particular, the differential operator
\begin{equation}\label{EllipticSig2Ope}
\square_h:=\sigma_2^{ij}(h)\nabla_{ij}^2
\end{equation} is elliptic, where $\sigma_2^{ij}(h)$ is defined in \eqref{DerSigEqu1}.

\begin{theorem}\label{IntEstMeaCurConLem}
Let $X : (M^n, g)\rightarrow \mathbb{R}^{n+1}$ be a convex hypersurface with positive scalar curvature. If there exists a unit constant vector $a$ such that $\langle X, a\rangle$ is a nonegative proper function, then we have
the interior estimate
\begin{equation}\label{IntEstMeaCurCon}
H(x)\leq C(n)\sup_{\{y \, | \,\langle X(y), a\rangle\leq 2\langle X(x),a\rangle \}}
(1+R^2(y)+ \frac{1}{R(y)}
+ \frac{1}{R^2(y)}|\nabla R|^2 (y)+ \frac{1}{R(y)}|\Delta R|(y) ).
\end{equation}
\end{theorem}

We note that on a shrinking GRS, one can split out lines and reduce the GRS to be a convex hypersurface such that there  exists  automatically a vector $a$ such that $\langle X, a\rangle$ is a nonnegative proper function, see \eqref{CompactSet}.

First of all, the following identity is well known (see, for example, (2.11) in \cite{CY77}) and will be used
to prove the theorem above.

\begin{lemma} Let $X: M^n \rightarrow \mathbb{R}^{n+1}$ be a hypersurface with the second fundamental form $h$, then \begin{equation}\label{EllEquMeaCur1}
\square \sigma_1:=\sigma_2^{ij}\sigma_{1,ij}=\Delta\sigma_2+|\nabla h|^2-|\nabla \sigma_1|^2+2\sigma_2|h|^2-(\sigma_1\sigma_2-3\sigma_3)\sigma_1,
 \end{equation}
 where $( \sigma_2^{ij}):= (\frac{\partial \sigma_2}{\partial h_{ij}})$ is defined in \eqref{DerSigEqu1}.
\end{lemma}

%Given a noncompact strictly convex hypersurface $(M,g)$, its second fundamental form $h$ is strictly %positive.
%There exists a unit vector $a\in R^{n+1}$ such that the set  $\Omega_{r}: = \{x \in M | \, \langle X(x),
%a \rangle\leq r \}$ is compact for all $r >0$, see \eqref{CompactSet}. In fact, $M$ is essentially a convex %graph
%along $a$ and $\langle X(x), a\rangle$ is
%asymptotic to the geodesic distance function $r(x)$ on $M$ (see, e.g., \cite{Wu74, CY77}).

To prove Theorem \ref{IntEstMeaCurConLem}, let $\phi(x)= r -\langle X(x), a\rangle$ be a cut off function with $r \geq 1$, we
will apply second derivative test to the auxiliary function
\[
\phi^2(x)\sigma_1(x)
\]
in the domain $\Omega_{r}: = \{x \in M | \, \langle X(x),
a \rangle\leq r \}$ to estimate $\sigma_1(x)$.

We may assume that $\phi^2 \sigma_1$ achieves its maximum at an interior point $\bar{x} \in \Omega_r$.
Let $0\leq \lambda_1(x)\leq \lambda_2(x)\leq \cdots \leq \lambda_n(x)$ be the principle
curvature of $M$ at $x \in M$. Moreover, in a neighborhood of $\bar{x}$, we choose a local orthonormal frame
$\{ e_i \}$ such that $ h_{ij}(\bar{x})=\lambda_i(\bar{x})\delta_{ij}$.

\vskip .1cm
We consider three cases.
\begin{enumerate}[label=Case (\Roman*)]
  \item : $\lambda_n(\bar{x})\leq \max\{n^2, 100n\}$. Then $\sigma_1(\bar{x})\leq C(n)$, and thus
\begin{equation}\label{IntEstMeaCurCaseI}
 \phi^2(\bar{x})\sigma_1(\bar{x})\leq C(n) r^2.
 \end{equation}
  \item : $\lambda_n(\bar{x})> \max\{n^2, 100n\}$ and $\lambda_{n-1}(\bar{x})\geq \lambda_n^{-\frac{1}{2}}
  (\bar{x})$, then the scalar curvature $R(\bar{x})=\sum_{i\neq j}\lambda_i\lambda_j\geq \lambda_n\lambda_{n-1} \geq \lambda_n^{\frac{1}{2}}(\bar{x})$ and $\sigma_1(\bar{x})=\sum_{i=1}^n\lambda_i\leq n\lambda_n\leq nR^2(\bar{x})$. Hence
\begin{equation}\label{IntEstMeaCurCaseII}
 \phi^2(\bar{x})\sigma_1(\bar{x})\leq nR^2(\bar{x}) r^2 .
 \end{equation}
 \item\label{MeaEstCaseIII} : $\lambda_n(\bar{x})> \max\{n^2, 100n\}$ and $\lambda_{n-1}(\bar{x}) < \lambda_n^{-\frac{1}{2}}
 (\bar{x})$. Then $\lambda_i(\bar{x}), i\neq n$ is much smaller than $\lambda_n(\bar{x})$. In this case, we have
       \begin{equation}\label{MeaMaxEigCom1}
       \lambda_i<\frac{1}{n^3}\lambda_n,\, \, i \neq n,\quad \hbox{and} \quad \lambda_n< \sigma_1=\sum_{i=1}^n\lambda_i<(1+\frac{1}{n^2})\lambda_n.
       \end{equation}
\end{enumerate}

To estimate $\phi^2(\bar{x})\sigma_1(\bar{x})$ from above, we are left to consider case \ref{MeaEstCaseIII}.
Note that the function
\[\zeta:=\ln\big(\phi^2\sigma_1\big)=2\ln\phi+\ln \sigma_1\]
on $\Omega_{r}$ also achieve the maximum at $\bar{x}$.
Apply the first and second derivative test, we get that at $\bar{x}$
\begin{equation}\label{GraAuxZer3}
0=\zeta_i(\bar{x})=2\frac{\phi_i}{\phi}+\frac{\sigma_{1,i}}{\sigma_1},  \quad \forall~ i=1, \cdots,n.
\end{equation}
and the matrix
\begin{equation}\label{HesAuxCon3}
\begin{split}
0&\geq \zeta_{ij}(\bar{x})\\
&=2\frac{\phi_{ij}}{\phi}-2\frac{\phi_i\phi_j}{\phi^2}+\frac{\sigma_{1,ij}}{\sigma_1}
-\frac{\sigma_{1,i}\sigma_{1,j}}{\sigma_1^2}\\
&=2\frac{\phi_{ij}}{\phi}+\frac{\sigma_{1,ij}}{\sigma_1}
-\frac{3}{2}\frac{\sigma_{1,i}\sigma_{1,j}}{\sigma_1^2},
\end{split}
\end{equation}
where we used \eqref{GraAuxZer3} in the last step.

Note that the positive scalar curvature on $M$ implies that the operator $\sigma_2^{ij}\nabla_{ij}^2$ is elliptic, see \eqref{EllipticSig2Ope}. Take the contraction of \eqref{HesAuxCon3} with $\sigma_2^{ij}$  and use \eqref{EllEquMeaCur1},
we get that at $\bar{x}$
\begin{equation}\label{ConAulMeaEqu1}
 \begin{split}
0&\geq  \sigma_2^{ij}\zeta_{ij} \\
&=2\frac{\sigma_2^{ij}\phi_{ij}}{\phi}+\frac{\Delta\sigma_2+|\nabla h|^2-|\nabla \sigma_1|^2+2\sigma_2|h|^2-(\sigma_1\sigma_2-3\sigma_3)\sigma_1}{\sigma_1}
 - \frac{3}{2} \frac{\sigma_2^{ij}\sigma_{1,i}\sigma_{1,j}}{\sigma_1^2}.
\end{split}
\end{equation}
Below we will deal with the three terms in the righthand side of \eqref{ConAulMeaEqu1} separately.
All the related calculations are at point $\bar{x}$.

For the first term, from \eqref{FirDerPosNorHyp}, we have
\begin{equation}\label{SquPsoConaFunEqu1}
\sigma_2^{ij}\phi_{ij}=-\sigma_2^{ij}\langle X_{ij},a\rangle=2\sigma_2\langle \nu, a\rangle.
\end{equation}

To deal with the second term, note that
\begin{equation}\label{GraSFFMeaDiff1}
\begin{split}
|\nabla h|^2-|\nabla \sigma_1|^2&=\sum_{i,j,k} h_{ij,k}^2-\sum_{i,j,k} h_{ii,k}h_{jj,k}\\
&=\sum_k\sum_{i\neq j} h_{ij,k}^2-\sum_k\sum_{i\neq j} h_{ii,k}h_{jj,k}\\
&=2\sum_{i\neq j} h_{ii,j}^2+\sum_{ i\neq j\neq k} h_{ij,k}^2-\sum_k\sum_{i\neq j} h_{ii,k}h_{jj,k},
\end{split}
\end{equation}
where we have used the Codazzi equation $h_{ij,k}=h_{ik,j}$ to get the last equality.
The term $2\sum_{i\neq j} h_{ii,j}^2$ in \eqref{GraSFFMeaDiff1} will play a crucial role in our estimate.
The term  $\sum_{ i\neq j\neq k } h_{ij,k}^2$ can be discarded.
However we need to control the negative term $-\sum_k\sum_{i\neq j} h_{ii,k}h_{jj,k}$.
In fact, by G{\aa}rding's theory on hyperbolic
polynomial (see  \cite[Lemma 3.2]{GLL} or \cite[Lemma 2.2]{GRW}),  we have
\begin{equation}\label{GraBelGraSca1}
-\sum_{i\neq j}h_{ii,k}h_{jj,k}\geq -\frac{1}{2}\frac{ |\sigma_{2,k}| ^2}{\sigma_2},
 ~\forall~ k=1,\cdots,n.
\end{equation}
 Therefore
\begin{equation}\label{GraSFFMeaDiff2}
|\nabla h|^2-|\nabla \sigma_1|^2\geq 2\sum_{i\neq j} h_{ii,j}^2-\frac{1}{2}\frac{|\nabla\sigma_2|^2}{\sigma_2}.
\end{equation}

We now handle the third term in \eqref{ConAulMeaEqu1}.
Since the principle curvatures $\lambda_i, \, i \neq n$, are very small compared with $\lambda_n$, thus
$\sigma_2^{nn}=\sum_{i=1}^{n-1}\lambda_i \leq \sigma_1$ is also very
small. We may use \eqref{GraAuxZer3} to substitute
the partial derivative of mean curvature along
the $i=n$ direction,  which in turn is bounded.
We compute that for any $\epsilon>0$
\begin{equation}\label{ConGraTermEst1}
\begin{split}
\frac{\sigma_2^{ij}\sigma_{1,i}\sigma_{1,j}}{\sigma_1^2}
&=\frac{\sigma_2^{nn}\sigma_{1,n}^2+\sum_{i=1}^{n-1}\sigma_2^{ii}\sigma_{1,i}^2}{\sigma_1^2} \\
& \leq 4 \frac{\sigma_2^{nn}\phi_{n}^2}{\phi^2}+\sigma_1\sum_{i=1}^{n-1}\Big(\frac{\sigma_{1,i}}{\sigma_1}
\Big)^2 \\
& =  4 \frac{\sigma_2^{nn} \phi_{n}^2}{\phi^2} + \frac{\sum_{i=1}^{n-1} \big ( h_{nn,i}+\sum_{
j=1}^{n-1} h_{jj,i}\big )^2} {\sigma_1}  \\
& \leq 4 \frac{\sigma_2^{nn} \phi_{n}^2}{\phi^2} + (1+\epsilon)\frac{\sum_{i=1}^{n-1} h_{nn,i}^2}{\sigma_1}
+(1+\frac{4} {\epsilon})\frac{\sum_{i=1}^{n-1} \big(\sum_{j=1}^{n-1}h_{jj,i}\big)^2}{\sigma_1}. \\
\end{split}
\end{equation}

To deal with the last term in \eqref{ConGraTermEst1}, we consider the scalar curvature $\sigma_2$.
For $i\leq n-1$,
\begin{equation}\label{GradScaCurGraSFF}
\sigma_{2,i}=\sum_{j=1}^n\sigma_2^{jj}h_{jj,i}=\sigma_2^{nn}h_{nn,i}+\sum_{j=1}^{n-1}(\sigma_1
-\lambda_j)h_{jj,i},
\end{equation}
where we have used Proposition \ref{DerSigRan1} to get $\sigma_2^{jj} = \sigma_1
-\lambda_j$ at $\bar{x}$.  As a consequence, we get
\begin{equation}\label{GraQuaTerEst1}
\begin{split}
\Big ( \sum_{j=1}^{n-1}h_{jj,i}\Big )^2&=\Big ( \frac{\sigma_{2,i}}{\sigma_1}-\frac{\sigma_2^{nn}}{\sigma_1}
h_{nn,i}+\sum_{j=1}^{n-1}\frac{\lambda_j}{\sigma_1}h_{jj,i}\Big )^2\\
&\leq 3\Big(\frac{(\sigma_{2,i})^2}{\sigma_1^2}+\frac{(\sigma_2^{nn})^2}{\sigma_1^2}h_{nn,i}^2+(n-1)
\sum_{j=1}^{n-1}\frac{\lambda_j^2}{\sigma_1^2}h_{jj,i}^2\Big)\\
&\leq 3\Big(\frac{(\sigma_{2,i})^2}{\sigma_1^2}+\frac{(\sigma_2^{nn})^2}{\sigma_1^2}h_{nn,i}^2
+\frac{2n}{\sigma_1^3}\sum_{j=1}^{n-1}h_{jj,i}^2 \Big),\\
\end{split}
\end{equation}
where we have used $\lambda_j^2\leq \lambda_n^{-1}< 2\sigma_1^{-1}$ for $j\leq n-1$ to get the last
 inequality (see \eqref{MeaMaxEigCom1}).

Use equation \eqref{GradScaCurGraSFF} again, we get that for $i \leq n-1$
\begin{equation}\label{GraQuaTerEst3}
h_{ii,i}=\frac{\sigma_{2,i}}{\sigma_2^{ii}}-\frac{\sigma_2^{nn}}{\sigma_2^{ii}}h_{nn,i} -\sum_{j=1, j
\neq i}^{n-1} \frac{\sigma_2^{jj}}{\sigma_2^{ii}}h_{jj,i}.
\end{equation}
It follows from \eqref{MeaMaxEigCom1} that $\sigma_1\geq \sigma_2^{ii}=\sigma_1-\lambda_i\geq \frac{n-1}{n}
\sigma_1$ for $i\leq n-1$. Then
\begin{equation}\label{GraQuaTerEst4}
\begin{split}
h_{ii,i}^2 &\leq 3\Big(\frac{\big(\sigma_{2,i}\big)^2}{(\sigma_2^{ii})^2}+\frac{(\sigma_2^{nn})^2}{(\sigma_2^{ii}
)^2}h_{nn,i}^2+(n-2)\sum_{j=1, j\neq i}^{n-1} \Big(\frac{\sigma_2^{jj}}{\sigma_2^{ii}}\Big)^2h_{jj,i}^2\Big)\\
& \leq 3\Big(\frac{n}{n-1}\Big)^2 \Big(\frac{\big(\sigma_{2,i}\big)^2}{\sigma_1^2}+\frac{(\sigma_2^{nn}
)^2}{\sigma_1^2}h_{nn,i}^2 + (n-2) \sum_{j=1, j \neq i}^{n-1}h_{jj,i}^2 \Big)\\
&\leq 6 \Big(\frac{\big(\sigma_{2,i}\big)^2}{\sigma_1^2}+\frac{(\sigma_2^{nn})^2}{\sigma_1^2}
h_{nn,i}^2+n\sum_{j=1, j \neq i}^{n-1} h_{jj,i}^2\Big).
\end{split}
\end{equation}

Combine \eqref{GraQuaTerEst1} and \eqref{GraQuaTerEst4} and use $\sigma_1 >n^2$, we have
\begin{equation}\label{GraQuaTerEst5}
\begin{split}
\Big ( \sum_{j=1}^{n-1}h_{jj,i}\Big )^2
&\leq 3\Big[\frac{(\sigma_{2,i})^2}{\sigma_1^2}+\frac{(\sigma_2^{nn})^2}{\sigma_1^2}h_{nn,i}^2+
\frac{2n}{\sigma_1^3}\Big(\sum_{j=1, j\neq i}^{n-1}h_{jj,i}^2 \\
&\quad +6\big ( \frac{\big(\sigma_{2,i}\big)^2}{\sigma_1^2}+\frac{(\sigma_2^{nn})^2}{\sigma_1^2}h_{nn,i}
^2+n\sum_{j=1, j\neq i}^{n-1}h_{jj,i}^2\big )\Big)\Big]\\
&\leq 21\Big[\frac{(\sigma_{2,i})^2}{\sigma_1^2}+\frac{(\sigma_2^{nn})^2}{\sigma_1^2}h_{nn,i}^2+\frac{2n^2}
{\sigma_1^3}\sum_{j=1, j\neq i}^{n-1} h_{jj,i}^2\Big)\Big].
\end{split}
\end{equation}

Plug \eqref{GraQuaTerEst5} into \eqref{ConGraTermEst1}, we have
\begin{equation}\label{ConGraTermEstSma2}
\begin{split}
\frac{\sigma_2^{ij} \sigma_{1,i} \sigma_{1,j}}{\sigma_1^2}
&\leq  4 \frac{\sigma_2^{nn} \phi_{n}^2}{\phi^2} + (1+\epsilon)\frac{\sum_{i=1}^{n-1} h_{nn,i}^2}{\sigma_1} \\
& \quad +(1+\frac{4}{\epsilon})\frac{21}{\sigma_1}\Big[\frac{ \sum_{i=1}^{n-1} (\sigma_{2,i})^2}{\sigma_1^2}
+\frac{(\sigma_2^{nn})^2}{\sigma_1^2} \sum_{i=1}^{n-1} h_{nn,i}^2+\frac{2n^2}{\sigma_1^3} \sum_{i=1}^{n-1}
\sum_{j=1,j \neq i}^{n-1} h_{jj,i}^2\Big)\Big]\\
&\leq 4 \frac{\sigma_2^{nn} \phi_{n}^2}{\phi^2} + (1+\delta)\frac{1}{\sigma_1}\sum_{i=1}^{n-1}
\sum_{j=1, j\neq i}^{n}h_{jj,i}^2+21(1+\frac{4}{ \epsilon})\frac{ \sum_{i=1}^{n-1} (\sigma_{2,i})^2}{\sigma_1^3},
\end{split}
\end{equation}
where
\begin{equation}\label{DefDeltaEpl1}
\delta: =\epsilon+21(1+\frac{4}{\epsilon})
\big(\frac{\sigma_2^{nn}}{\sigma_1}\big)^2+21(1+\frac{4}{\epsilon})\frac{2n^2}{\sigma_1^3}.
\end{equation}

Put \eqref{SquPsoConaFunEqu1}, \eqref{GraSFFMeaDiff2} and \eqref{ConGraTermEstSma2} into \eqref{ConAulMeaEqu1},
we get
\begin{equation} \label{ConAulMeaEqu2}
 \begin{split}
0 &\geq \frac{4\sigma_2\langle\nu,a\rangle}{\phi}+\frac{\Delta\sigma_2}{\sigma_1}+
\frac{1}{\sigma_1}\big(2\sum_{i\neq j} h_{ii,j}^2-\frac{1}{2}\frac{|\nabla\sigma_2|^2}{\sigma_2}\big)
+\Big( \frac{2 |h|^2}{\sigma_1^2}-1\Big)\sigma_2\sigma_1\\
&\quad +3\sigma_3- \frac{3}{2}\Big(\frac{4\sigma_2^{nn}\phi_{n}^2}{\phi^2}+(1+\delta)\frac{1}{\sigma_1}
  \sum_{i \neq j}h_{ii,j}^2+21(1+\frac{4}{\epsilon})\frac{|\nabla\sigma_{2}|^2
 }{\sigma_1^3}\Big)\\
&\geq \frac{4 \sigma_2\langle\nu,a\rangle}{\phi}+\frac{\Delta\sigma_2}{\sigma_1}
+\Big(2- \frac{3}{2} (1+\delta)\Big) \frac{1}{\sigma_1} \sum_{i \neq j }h_{ii,j}^2-\frac{|\nabla\sigma_2|^2}{2\sigma_2\sigma_1} \\
&\quad +\Big(\frac{2|h|^2}{\sigma_1^2} -1\Big)\sigma_2\sigma_1 - \frac{6\sigma_2^{nn}\langle e_n,a
\rangle^2}{\phi^2}- \frac{63}{2}(1+\frac{4}{\epsilon})\frac{|\nabla\sigma_{2}|^2}{2\sigma_1\sigma_2},
\end{split}
\end{equation}
where we have used the Newton-MacLaurin inequality $\sigma_1^2\geq\frac{2n}{n-1}\sigma_2\geq 2\sigma_2$.

By \ref{MeaEstCaseIII} assumption and \eqref{MeaMaxEigCom1} we have
 \begin{equation}\label{StrPosCoeMea1}
\frac{2|h|^2}{\sigma_1^2} -1 \geq \frac{2 \lambda_n^2}{\sigma_1^2} -1 \geq \frac{2n^2}{n^2+1} -1  \geq \frac{1}{2}.
\end{equation}
Again by \eqref{MeaMaxEigCom1} we have
\[
\sigma_2^{nn}=\sum_{i=1}^{n-1}\lambda_i<n\lambda_n^{-\frac{1}{2}}<\frac{2n}{\sqrt{\sigma_1}}.
\]
Take $\epsilon=\frac{1}{10}$ in \eqref{DefDeltaEpl1} and use $\sigma_1\geq 100 n$, we get
\begin{equation}\label{NeaCalGraSFF1}
\delta \leq \frac{1}{10}+1000 \frac{4n^2}{\sigma_1^3}+1000\frac{2n^2}{\sigma_1^3} <\frac{1}{3}.
\end{equation}

Hence it follows from \eqref{StrPosCoeMea1}, \eqref{NeaCalGraSFF1} and \eqref{ConAulMeaEqu2} that in Case III
at point $\bar{x}$
\begin{equation}\label{MeaCurIntEst2}
\begin{split}
 (\phi^2\sigma_1)&\leq \frac{2}{\sigma_2}\Big[-4\sigma_2\langle \nu, a\rangle\phi+ \frac{12n}
{\sqrt{\sigma_1}} \langle e_n,  a\rangle^2 +\Big(-\frac{\Delta\sigma_2}{\sigma_1}+1000\frac{|\nabla\sigma_2|^2}{\sigma_2\sigma_1}
\Big)\phi^2\Big].
\end{split}
\end{equation}

Combine the three cases \eqref{IntEstMeaCurCaseI}, \eqref{IntEstMeaCurCaseII} and \eqref{MeaCurIntEst2}
 all together and use $H(x)=\sigma_1(x)$ and $R(x) =2\sigma_2(x)$, we have proved the interior estimate
 of the mean curvature
\begin{equation}
\sup_{x \in \Omega_{r}} \phi^2(x) H(x)  \leq  C(n) \sup_{y \in \Omega_{r}} (1+R^2(y)+ \frac{1}{R(y)}
+ \frac{1}{R^2(y)}|\nabla R|^2 (y)+ \frac{1}{R(y)}|\Delta R|(y) ) r^2. \label{eq add 1}
\end{equation}

Given any $x \in M$, we take $r =2\langle X(x),a\rangle$,
then $\phi(x)= \langle X(x),a\rangle$.  Hence by \eqref{eq add 1},  we obtain \eqref{IntEstMeaCurCon}. Theorem \ref{IntEstMeaCurConLem} is proved.

\section{The Structure of Generalized Cylinder}\label{StrGenCyl}
In this section, we prove Theorem \ref{GSRSHypCyl}.
First we define an operator which generalizes Cheng-Yau's self-adjoint operator associated to
a given Codazzi tensor \cite{CY77}.
\begin{proposition}\label{SlfAdjDifOpe}
Let $(M^n, g)$ be a Riemannian manifold and let $f$ be a smooth function on $M$.
Let $\psi=\sum_{ij}\psi_{ij} \omega^i \omega^j$ be a symmetric $(2,0)$-tensor on $M$.
Then the operator
\[\square_{\psi,f} :=\psi_{ij} \nabla^2_{ij}-\psi_{ij} f_i \nabla_j\]
is self-adjoint with respect the weighted measure $e^{-f}d \mu$ if and only if
the divergence of $\psi$ $\sum_j\psi_{ij,j}=0$ for all $i$. Here $\nabla$ is the Levi-Civita
connection  and $d \mu$ is the volume element associated with $g=\sum_i\omega_i^2$.

In particular, given a symmetric Codazzi tensor $h= h_{ij}\omega^i\omega^j$ with
$h_{ij,k}= h_{ik,j}$, define $\sigma_2(h)$, $\sigma_2^{ij}(h)$ as in \eqref{DefSigk} and \eqref{DerSigEqu1}, then the tensor $(\psi(h))_{ij} :=\big(\sum_k h_{kk}\big) \delta_{ij}-
h_{ij}=\sigma_2^{ij}(h)$ is divergence free and the operator
\begin{equation}\label{GenCYOpeDef1}
\begin{split}
\square_{\psi(h),f}
&=\sigma_2^{ij}(h) \nabla^2_{ij}-\sigma_2^{ij}(h)f_i \nabla_j,
\end{split}
\end{equation}
is a self-adjoint operator with respect to the weighted measure $e^{-f}d\mu$.
\end{proposition}

\begin{proof}
For any compact supported $C^2$ functions $u$ and $v$ on $M$, by Stoke's theorem,
\[\begin{split}  \int_{M} \square_{\psi,f} u \cdot v e^{-f}d\mu
&=\int_{M} (\psi_{ij}u_{i}ve^{-f})_j d\mu-\int_{M} \psi_{ij}u_i
v_je^{-f} d\mu
-\int_{M} \psi_{ij,j}u_ive^{-f} d\mu \\
& = -\int_{M}\psi_{ij}u_iv_je^{-f} d\mu.
\end{split}\]
Therefore the operator $\square_{\psi, f}$ is self-adjoint with respect to the measure
$e^{-f}d\mu$.
\end{proof}

\begin{remark}  On a GRS satisfying \eqref{GraShrRicciSolitonEqu1}, there is another natural intrinsic
self-adjoint operator $\square_{\operatorname{Ric}}=R_{ij}\nabla_{ij}^2 $
with respect the weighted measure $e^{-f}d\mu$. We believe that this operator should be useful.
\end{remark}

We need the following proposition (analogue of \cite[Proposition 2]{CY77}).
\begin{proposition}\label{PinTypIneSqu1}
Let $(M^n,g)$ be a compact manifold with boundary and let $f$ be a smooth function
on $M$. Suppose $\psi=\sum_{ij}\psi_{ij} \omega^i \omega^j$ is a semipositive symmetric $(2,0)$-tensor
which is divergence free. Then the (possibly degenerate) elliptic operator
 $\square_{\psi, f}$ has the following
 property. For any $C^2$ positive function $u$ and any non-negative $C^2$ function $v$ satisfying
 $v|_{\partial M}=0$, we have
\begin{equation}\label{FirEigSqu1}
\Big(-\int_M v\square_{\psi, f} ve^{-f} d \mu \Big) \Big(\int_Mv^2e^{-
f} d\mu \Big)^{-1}\geq \inf_M\frac{-\square_{\psi, f} u}{u}.
\end{equation}
\end{proposition}

\begin{proof}
We only need to prove \eqref{FirEigSqu1} assuming $\square_{\psi, f}$ is non-degenerate
elliptic, as one may
replace $\square_{\psi, f}$ by $\square_{\psi, f}+\epsilon \Delta_{f}$
and let $\epsilon\rightarrow 0$.

Let $\lambda$ be the first (positive) eigenvalue and $v_\lambda$ be the first eigenfunction of
$\square_{\psi,  f}$ over $M$ with the zero boundary condition.
Then it is well-known that the left hand of \eqref{FirEigSqu1} is always not less than $\lambda$ and
$v_\lambda$ is positive in the interior of $M$.

Consider the function $\frac{v_\lambda}{u}$. At the interior point where  $\frac{v_\lambda}{u}$ attains
its maximum, we have
\begin{align}
& 0= \nabla \frac{v_\lambda}{u}=\frac{u \nabla v_\lambda-v_\lambda \nabla
u}{u^2}, \label{GraEigFunQuo1} \\
& 0\geq \nabla^2\frac{v_\lambda}{u}=\frac{u \nabla^2 v_\lambda-v_\lambda \nabla^2
u}{u^2}. \label{HessEigFunQuo1}
\end{align}
Hence
 \[\begin{split}
\lambda&=\frac{-\square_{\psi,f} v_\lambda}{v_\lambda}=-\frac{\psi_{ij}\nabla^2_{ij} v_\lambda}{v_\lambda}+\frac{\psi_{ij}f_i \nabla_jv_\lambda}{v_\lambda}\\
&\geq-\frac{\psi_{ij} \nabla^2_{ij} u}{u}+\frac{\psi_{ij} f_i \nabla
_ju}{u}= \frac{-\square_{\psi, f}u}{u}.
\end{split}
\]
This verifies \eqref{FirEigSqu1}.
\end{proof}

Let $X: M^n \rightarrow \mathbb{R}^{n+1}$ be a hypersurface with induced metric $g$ and (outer) unit normal $\nu$.
Let $h$ be the second fundamental form and let $H$ be the mean curvature.
Given a smooth function $f$ on $M$, we define two operators
using a local orthonormal frame $\{e_1,\cdots,e_n\}$
\begin{equation*}
\square_h := \sigma_2^{ij}(h) \nabla_{ij}^2 \quad \text{ and } \quad
\square_{h,f}: = \square_h - \sigma_2^{ij}(h) f_i \nabla_j,
\end{equation*}
where $\sigma_2(h)$, $\sigma_2^{ij}(h)$  are defined in \eqref{DefSigk} and \eqref{DerSigEqu1}.
From Proposition \eqref{SlfAdjDifOpe}, the operator $\square_{h,f}$ is a self-adjoint operator with respect to the weighted measure $e^{-f}d\mu$.

We compute $\square_{h,f}\nu$. Note that
\begin{equation}\label{DerPosNorHyp1}
X_{ij} = -h_{ij}\nu, \quad \nu_i = h_{il}e_l, \quad
\nu_{ij}= h_{ij,l}e_l-h^2_{ij}\nu, \qquad 1\leq i, j\leq n.
\end{equation}
From \eqref{DerPosNorHyp1}, we have
\begin{eqnarray}
%\square_h X&=&(H\delta_{ij}-h_{ij})(-h_{ij}\nu) =-R \nu, \label{SquPsoFunEqu0}\\
  \square_{h,f} \nu &=&\sigma_2^{ij}(h)(h_{ij,l}e_l-h^2_{ij}\nu)-\sigma_2^{ij}(h)f_ih_{jl})e_l\nonumber\\
   &=&\frac{1}{2}\nabla R-(\sigma_1(h)\sigma_2(h)-3\sigma_3(h))\nu- \operatorname{Ric}(\nabla f). \label{SquNorVecEqu1}
\end{eqnarray}

For a GRS satisfying \eqref{GraShrRicciSolitonEqu1}, we have the following equation analog to the one considered
by Cheng-Yau in \cite{CY77}.
\begin{proposition}\label{SquNorVecEquGSRSPro}
Let $(M^n,g,f)$ be an isometrically embedded hypersurface in $\mathbb{R}^{n+1}$ with a GRS structure \eqref{GraShrRicciSolitonEqu1}. Then
\begin{equation}\label{SquNorVecEquGSRS1}
  \square_{h,f}\nu  = -(\sigma_1(h)\sigma_2(h)-3\sigma_3(h))\nu.
\end{equation}
\end{proposition}

\begin{proof}
The proposition follows from \eqref{SquNorVecEqu1} and \eqref{RicSolIde1}.
\end{proof}

\smallskip
Now we can prove Theorem \ref{GSRSHypCyl}.  Since $(M,g,f)$ be a gradient shrinking Ricci soliton
with nonnegative Ricci curvature,  then by the maximum principle,  either $(M,g,f)$ is flat or the scalar curvature is strictly positive (see \eqref{ScaStrPosBou1}).
Suppose $(M,g)$ is an Euclidean hypersurface with nonnegative Ricci curvature and positive scalar curvature, then the second fundamental form $h$ is semipositive.
It follows that $M$ is a convex hypersurface and the curvature operator is semipositive.

With the splitting Theorem \ref{SplGSRSThm1}, we can assume that the universal covering  $(\widetilde{M},\tilde{g})=(N,h)\times \mathbb{R}^{n-k}$ split isometrically and $(N,h)$ has
   strictly positive Ricci curvature. However, it is not illuminating to see that whether $(N,h)$ or one of its quotient admit an isometric embedding in $\mathbb{R}^{k+1}$.  Alternatively, since $(M,g)$ is Euclidean hypersurface, we can also establish an splitting theorem easily in an extrinsic way.

First of all, if $M$ is noncompact convex  Euclidean hypersurface, then the Gauss image must lies in a closed hemisphere of $S^n$; moreover, if $(M,g)$ has positive section curvature, then Gauss image is an open convex subset of $S^n$ \cite{Wu74}. Therefore, there is a unit vector $a\in S^n$ so that $\langle \nu, a\rangle\geq 0$ on $M$. If $\langle \nu, a\rangle=0$ at one point, then we claim that $\langle \nu, a\rangle$ is identically zero. From \eqref{SquNorVecEquGSRS1} we get
\begin{equation}
\square_{h,f}\langle \nu, a\rangle=-(\sigma_1(h)\sigma_2(h)-3\sigma_3(h))\langle \nu, a\rangle.
\end{equation}
Note that the positive scalar curvature on $M$ implies that the operator $\square_{h,f}$ is elliptic, see \eqref{EllipticSig2Ope}.
Hence by applying the maximum principle we  conclude that either $\langle \nu, a\rangle$ is everywhere positive or $\langle \nu, a\rangle\equiv 0$. In the later case, since $a$ is constant tangent vector and therefore parallel, we can split out one line  globally along $a$. Continue by induction, we prove that $M^n=N\times \mathbb{R}^{n-k}$ for some $2\leq k\leq n$, where  $N$ does not contain any straight lines.

We note that on a codimension one shrinking GRS isometrically embedded in $\mathbb{R}^{n+1}$,  we have the equation \eqref{RicSolIde2}
\[\sigma_2(h)=\frac{1}{2}R=\frac{1}{2}(f-|\nabla f|^2).\]
In the same way as did for the Ricci curvature in Theorem \ref{SplGSRSThm1}, we can deduce the constant rank theorem for the second fundamental form, and therefore the splitting theorem.

In the following we will show by contradiction argument that $N$ must be compact.
With the splitting structure, let us assume $M(=N)$ is noncompact and there exists a vector $a$ such that $\langle\nu, a\rangle>0$. In this case,  $M$ is essentially a graph along $-a$, the set \begin{equation}\label{CompactSet}
\Omega_{r}=\{x \in M |\langle X(x),
-a\rangle\leq r \}
\end{equation}
 is compact for all $r>0$ and $\langle X(x),
-a\rangle$ is asymptotic to the geodesic distance of $M$,  see \cite{Wu74, CY77}.

Combine Proposition \ref{PinTypIneSqu1} and Proposition \ref{SquNorVecEquGSRSPro}, we have
that for any compact region $\Omega \subset M$ and for any nonnegative $C^2$ function $u$ with
$u\big|_{\partial \Omega}=0$
\begin{equation}\label{FirEigNorEqu}
\begin{split}
\min_{\Omega}\Big(\sigma_1\sigma_2-3\sigma_3\Big)&= \min_\Omega \Big(\frac{-\square_{h,f} \langle\nu, a\rangle}{\langle\nu, a\rangle}\Big)\\
&\leq \Big(-\int_\Omega u\square_{h,f}ue^{-f}\Big)\Big(\int_\Omega u^2e^{-f}\Big)^{-1}\\
&=\Big(\int_\Omega\sigma_2^{ij}u_iu_je^{-f}\Big)\Big(\int_\Omega u^2e^{-f}\Big)^{-1}.\\
\end{split}
\end{equation}

We apply \eqref{FirEigNorEqu} to $u(x)= r
-\langle X(x), -a\rangle$ on $\Omega_{r}$ and get
\begin{equation}\label{FirEigPosVec}
\begin{split}
\Big(&\int_{\Omega_{r}}\sigma_2^{ij}u_iu_je^{-f}\Big)\Big(\int_{\Omega_{r}} u^2e^{-f}\Big)^{-1}\\
&=\Big(\int_{\Omega_{r}}(H\delta_{ij}-h_{ij})\langle e_i,a\rangle\langle e_j,a\rangle e^{-f}\Big)\Big(\int_{\Omega_{r}}(r-\langle X,-a\rangle)^2e^{-f}\Big)^{-1}\\
&\leq \Big(\int_{\Omega_{r}}He^{-f}\Big)\Big(\frac{r^2}{4}\int_{\Omega_{\frac{r}
{2}}}e^{-f}\Big)^{-1}\\
&= 4r^{-2}\Big(\int_{\Omega_{r}}He^{-f}\Big)\Big(\int_{\Omega_{\frac{r}{2}}}e^{-f}\Big)^{-1}.
\end{split}
\end{equation}

Since $(M,g)$ is a  shrinking GRS  with positive Ricci curvature, combine the
fact that the scalar curvature have a strictly lower bound  \eqref{ScaStrPosBou1} with the scalar curvature growth estimate  \eqref{PloGroScaCur1}, \eqref{PloGroGraScaCur1}
and \eqref{PolyGrowLapSca}, \eqref{IntEstMeaCurCon} implies that
\begin{equation}\label{meaIntWeiMea1}
H(x)\le C(n)(1+r(x))^4,
\end{equation}
where we used the fact that $\langle X(x), -a \rangle$ is asymptotic to the intrinsic geodesic distance function
$r(x)$ of $M$.  In particular, the mean curvature is integrable with respect to the weighted measure $e^{-f}d\mu$,
\begin{equation}\label{MeaIntWeiMea1}
\int_M H e^{-f}d \mu<\infty.
\end{equation}
Combine with \eqref{FirEigNorEqu}, \eqref{FirEigPosVec} and \eqref{MeaIntWeiMea1}, we have
\begin{equation}\label{NonComAloCyl1}
\min_{\Omega_{r}}\Big(\sigma_1\sigma_2-3\sigma_3\Big)\leq C(n)r^{-2}.
\end{equation}
Let $r$ go to infinity, this implies that
\begin{equation}\label{NonComAsyCyl1}
\inf_{M}\Big(\sigma_1\sigma_2-3\sigma_3\Big)=0.
\end{equation}
On the other hand, since $h$ is in the so-called G{\aa}rding cone $\Gamma_2$, by Newton-MacLaurin inequality we have
\begin{equation}
\sigma_1\sigma_2-3\sigma_3\geq C(n)\sigma_2^{\frac{3}{2}}>\delta,
\end{equation}
where the uniform positive lower bound of scalar curvature is ensured by \eqref{ScaStrPosBou1}.
This contradicts with \eqref{NonComAsyCyl1}, and consequently, $M$ must be compact.

Since $M$ is compact hypersurface, then Ricci tensor has full rank at the elliptic point, i.e.  there exist $x_0\in M$ such that ${\rm Ric}(x_0)>0$. With Theorem \ref{SplGSRSThm1}, we conclude that Ricci tensor has constant rank; therefore the Ricci curvature and also the curvature operator are positive everywhere.  Finally, by Theorem \ref{GSRSPosCOCla}, the compact  shrinking GRS  with positive curvature operator has to be the round sphere or its metric quotient. Since $M$ is an Euclidean hypersurface, then it must be a round sphere. Now the proof of Theorem \ref{GSRSHypCyl} is complete.

\end{document}